\documentclass[12pt]{amsart}
\usepackage{verbatim}
\usepackage[utf8]{inputenc}
\usepackage{amsmath,amssymb,amsthm}
\usepackage{tikz-cd}
\usepackage{pgfplots}
\pgfplotsset{compat=1.11}
\usepgfplotslibrary{fillbetween}
\usetikzlibrary{intersections}
\usepackage{bbm}
\usepackage{hyperref}  
\usepackage{comment}
\usepackage{stmaryrd}
\usepackage[new]{old-arrows}
\usepackage{graphicx}
\graphicspath{ {./Images/} }
\usepackage{todonotes}
\usepackage{booktabs}
\usepackage [english]{babel}
\usepackage [autostyle, english = american]{csquotes}
\MakeOuterQuote{"}

\numberwithin{equation}{section}

\newtheorem{theorem}{Theorem}[section]
\newtheorem*{main theorem}{Main Theorem}

\theoremstyle{definition}
\newtheorem*{definition}{Definition}
\newtheorem{remark}[theorem]{Remark}
\newtheorem{example}[theorem]{Example}

\newtheorem{proposition}[theorem]{Proposition}
\newtheorem{lemma}[theorem]{Lemma}
\newtheorem{corollary}[theorem]{Corollary}
\newtheorem{question}[theorem]{Question}
\newtheorem{conjecture}[theorem]{Conjecture}
\newtheorem*{theorem*}{Theorem}

\newcommand{\N}{\mathbb{N}}
\newcommand{\Z}{\mathbb{Z}}
\newcommand{\R}{\mathbb{R}}

\newcommand{\dimH}{\dim_H}

\newcommand{\dimBu}{\mathop{\overline{\dim}_B}}
\newcommand{\dimBl}{\mathop{\underline{\dim}_B}}
\newcommand{\mF}{\mathcal{F}}
\newcommand{\Freq}{\mathrm{Freq}}

\newcommand{\Cantor}{J}
\newcommand{\Dist}{\varkappa}
\newcommand{\Root}{\mathop{\mathrm{Root}}}

\newcommand{\bc}{\overline{c}}

\newcommand{\eps}{\varepsilon}
\newcommand{\osc}{\mathop{\mathrm{osc}}}

\newcommand{\prm}{\mathbf{a}}
\newcommand{\ha}{\tau}
\newcommand{\fa}{\alpha}

\pgfdeclarelayer{bg}
\pgfsetlayers{bg,main}

\title[Regularity of the Dimension]{On the Regularity of the Dimension of Cookie-Cutter-Like Sets}
\author{Victor Kleptsyn}
\thanks{VK was supported by ANR Gromeov (ANR-19-CE40-0007) and by Centre Henri Lebesgue (ANR-11-LABX-0020-01).}
\address{CNRS, IRMAR (UMR CNRS 6625), University of Rennes, France}
\email{victor.kleptsyn@univ-rennes.fr}
\author{Alexandro Luna}
\thanks{AL was supported by NSF grant DMS-2247966 (PI: A. Gorodetski)}
\address{Department of Mathematics, University of California, Irvine, CA 92697, USA}
\email{lunaar1@uci.edu}
\date{}

\begin{document}
\maketitle

\begin{abstract}
   We study the Hausdorff and box-counting dimensions of cookie-cutter-like sets formed by sequential dynamics of a finite number of expanding maps. Under some natural conditions, these dimensions turn out to be the minimum and maximum of the corresponding dimensions of the cookie-cutter sets generated by the individual expanding maps. In the case of one-parameter families of such systems, this provides a simple mechanism for producing non-differentiable fractal dimensions as functions of the parameter. This supports a conjecture that the Hausdorff dimension of the spectrum of a Sturmian Hamiltonian, in general, does not have to be differentiable as a function of the coupling constant. This is in drastic contrast to the analytic dependence of the dimensions of such spectra with quadratic irrational frequencies, e.g. the Fibonacci Hamiltonian, previously shown by M.~Pollicott.
\end{abstract}

\section{Introduction}

This paper is devoted to the study of the dimension of cookie-cutter-like sets and their dependence on the generating transformations. More specifically, we consider sets that are constructed using a sequence of expanding maps, chosen from a finite set. Under some reasonable conditions that affect only the sequence of indices, but do not depend on the maps themselves, the Hausdorff dimension of such cookie-cutter-like set turns out to be the minimum of the dimensions of the associated individual cookie-cutter sets. Under the same assumptions, the upper box counting dimension is equal to the maximum; see Theorem~\ref{t:main} below.

As the conditions depend only on the sequence of indices, but not on the maps involved, the conclusion generalizes to the setting when the maps depend on a parameter. This allows to obtain a simple mechanism of loss of differentiability of the dimension as a function of the parameter: the minimum (or the maximum) of differentiable functions can fail to be differentiable at any point where their graphs intersect. This statement is closely related to the non-differentiability result obtained in~\cite{MSU2013} for sequences of holomorphic maps depending on a sequence of parameters (that is a point of $\ell_{\infty}$ in their case). 

Our motivation comes from the study of the regularity of the dimension of the spectrum of a Sturmian Hamiltonian as a function of the coupling constant. In general, it is known that the Hausdorff and upper box-counting dimensions of such spectra are (locally) Lipschitz as functions of the coupling constant, in the regime when this constant is large enough, for any irrational frequency~\cite{LQW2014}. For the class of Sturmian Hamiltonians with frequencies that are quadratic irrationalities, the regularity is even known to be analytic (a corollary of hyperbolic results for the corresponding trace maps~\cite{C1986,DG2009,Ca2009}, joined with the analytic dependence results in the stationary hyperbolic setting~\cite{P2015}). However, as we explain in Section~\ref{s:discussion}, our results motivate the conjecture that for a general frequency (of bounded type), the dimension of such spectra may not even be differentiable as a function of the coupling constant. 

\subsection{Background}

Iterated function systems (IFSs for short) and their limit sets have been studied by many authors for a very long time. One standard setting is that one takes a finite number of contracting maps $(f_1,\dots,f_q)$ of a compact metric space $X$ to itself; then (see~\cite{Hutchinson}) there is a unique set $K$ such that 
\[
K=\bigcup_{i=1}^q f_i(K),
\]
and one can find this set by 
\[
K=\bigcap_n X_{(n)}, \, \text{ where } X_{(0)}
:=X, \quad X_{(n)}=\bigcup_{i=1}^q f_i(X_{(n-1)}), \quad n=1,2,\dots.
\]

If the maps $f_i$ are homeomorphisms onto their images, and their images $f_i(X)$ are pairwise disjoint, then the limit set is a Cantor set. Throughout this paper, we will assume that the set $X$, on which all these maps are defined, is the closed interval~$I$. (Note that some of the maps $f_i:I\to I$ might be orientation-reversing: we do not impose the preservation of orientation condition.)

The same setting admits another point of view, the one of cookie-cutter maps:
\begin{definition}
    Assume that $I_1,\dots,I_q\subset I$ are disjoint subintervals. A map $F:\bigcup_{i=1}^q I_i\to I$, for which all the restrictions $F:I_i\to I$ are expanding diffeomorphisms, is called a \emph{cookie-cutter map}. 
    The Cantor set 
    \begin{equation}\label{eq:C-F}
        \Cantor(F):= \left\{ x\in I \mid \forall n \quad F^n(x) \text{ is defined} \right\}
    \end{equation}
    is called the corresponding \emph{cookie-cutter set}. 
\end{definition}

The equivalence between the two settings comes from considering the cookie-cutter map~$F$, corresponding to an IFS $(f_1,\dots,f_q)$, given by
\[
F|_{I_i}= f_i^{-1}, \quad i=1,\dots, q;
\]
see Fig.~\ref{fig:cookie-cutter}.

\begin{figure}
    \centering
    \mbox{} \hfill 
    \includegraphics[width=0.3\linewidth]{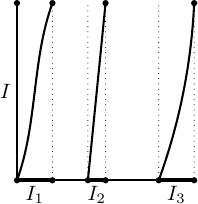} \hfill
    \includegraphics[width=0.3\linewidth]{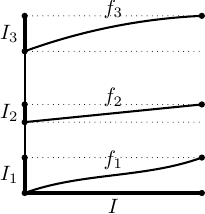}
    \hfill \mbox{}
    \caption{Left: cookie-cutter map~$F$; right: the maps $f_j$, forming the corresponding IFS, and the intervals $I_j=f_j(I)$.}
    \label{fig:cookie-cutter}
\end{figure}

A standard assumption in this setting, that we will be using throughout this paper, is that all the maps $f_i$ are of class~$C^{1+\ha}$ for some $\ha>0$. This assumption allows to conclude that the \emph{distortion} of all the compositions 
\[
f_{\omega}:=f_{\omega_1}\circ \dots \circ f_{\omega_n}, \quad \omega_1,\dots,\omega_n \in \{1,\dots,q\}
\]
is uniformly bounded. In particular, under this assumption, the Hausdorff, upper and lower box-counting  dimensions of the limit set~$K=\Cantor(F)$ coincide (see, e.g., \cite[Chapter~4]{PT1993} and references therein): 
\[
\dimH K= \dimBu K =\dimBl K.
\]
Given this equality, for cookie-cutter sets we will be denoting this dimension by simply~$\dim K$.

The study of cookie-cutters is also motivated by the study of stable laminations for basic hyperbolic invariant sets with one-dimensional unstable manifolds. Namely, one takes an intersection of a stable lamination with a transverse interval~$I$, then takes a sufficiently large iterate of it by the dynamics and then projects back to~$I$ by the holonomy along the stable lamination. Under reasonable assumptions, the Cantor set defined by the intersection of $I$ with the stable lamination of the hyperbolic set, is a $C^{1+\ha}$-dynamically defined Cantor set  
(again, see Chapter~4 of \cite{PT1993} for a discussion). 

\begin{figure}[h]
                \centering
\includegraphics[scale=0.9]{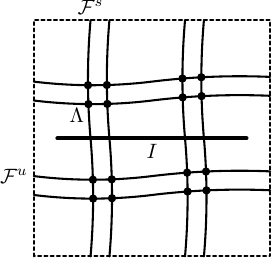} \quad 
\includegraphics[scale=0.9]{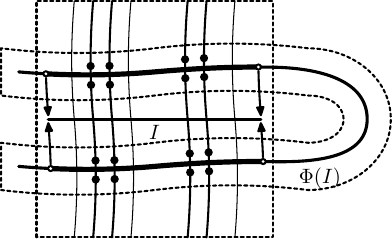}
                \caption{Left: basic set $\Lambda$ of a horseshoe map $\Phi$ with its stable and unstable laminations, and a transverse interval $I$; right: an iteration of this interval and its projection along the stable foliation that forms a cookie-cutter}\label{Figure: stable holonomy}
            \end{figure}

Our paper is devoted to the study of \emph{cookie-cutter-like} sets, corresponding to \emph{non-stationary} dynamics. Namely, instead of applying the same IFS (or the same expanding map) on each step, we fix a finite set $\mF=\{F_1,\dots,F_k\}$ of cookie-cutter maps, a sequence 
\[
\bc:=(c_n)_{n\in\N}\in\{1,\dots,k\}^{\N},
\]
and study the associated cookie-cutter-like Cantor set  defined by 
\begin{equation}\label{eq:non-st-Cantor}
\Cantor(\mF;\bc):= \{x\in I \mid \forall n \quad (F_{c_n}\circ\dots\circ F_{c_1})(x) \, \text{ is defined}\}.    
\end{equation}
When the set $\mF$ is already fixed, we will also denote this set by a shorter notation~$\Cantor_{\bc}$.

If $(f_{i,1},\dots,f_{i,q_i})$ is the IFS associated to $F_i$, this set can be also defined as the intersection 
\begin{equation}\label{eq:K-c-n}
\Cantor_{\bc}=\bigcap_n K_{c_1\dots c_n}, \quad K_{\emptyset}=I, \quad K_{c_1\dots c_n}= \bigcup_{i=1}^{q_{c_1}} f_{c_1,i}(K_{c_2\dots c_n}).    
\end{equation}
Such systems are considered, for instance, in~\cite{RU2016, MRW2001}. 

\subsection{Main results and plan of the paper}

Our main result, Theorem~\ref{t:main}, which we state in Section~\ref{s:rarely} below, provides a description of the Hausdorff and upper box-counting dimensions for cookie-cutter-like sets, associated to \emph{rarely switching} systems, satisfying the \emph{frequencies condition} (see the definition below). Applying it then in Section~\ref{s:parameter} to the study of ($C^2$ or even analytic) parameter-dependent systems, we see that such systems can manifest non-differentiable dependence of these dimensions on the parameter (contrary to what happens for stationary systems, where the dependence is known to be differentiable, see~\cite{R1982,Ma1990}).

This result, in a sense, is a ``proof of concept''; we further discuss the motivation for such study in Section~\ref{s:discussion}. Namely, cookie-cutter sets and their dimensions are related to the stable dimension of locally maximal hyperbolic invariant sets. In their turn, the latter occur in the study of the spectral properties of the quasi-periodic discrete Schr\"odinger operator with a Sturmian potential. Due to this link and analogy, the dependence of the dimension of the spectrum~$\Sigma_{\fa,\lambda}$ of such an operator on the \emph{coupling constant~$\lambda$} is known to be analytic when~$\fa$ is a quadratic irrationality. However, the arguments that we present in Section~\ref{s:discussion} allow to conjecture that (if all the statements generalize as we expect them to, and some additional arguments work) such a dependence can be non-differentiable even for frequencies~$\fa$ of a bounded type; see Conjecture~\ref{conj:quadratic} therein.

In Section~\ref{s:Moran}, we recall the Moran formula, thermodynamic formalism, and the distortion control technique. Once this is done, we pass to the proofs of the main results in Section~\ref{s:proofs}. The key moment of the proof is Proposition~\ref{p:combine}, which allows us to approximate a non-stationary pressure function of a rarely switching system, by averages of stationary ones. We state and prove it, and deduce from this proposition the main results. We conclude the paper by showing in Section~\ref{s:lower}, that the rare switching condition in Theorem~\ref{t:main} cannot be omitted; see Example~\ref{ex:HD-small} therein.

\subsection{Acknowledgments}

The authors are very grateful to Anton Goro\-detski for proposing the problem and encouraging discussions, and to Jake Fillman for many helpful remarks.

\section{Main results}
\subsection{Rarely switching systems}\label{s:rarely}

The first main results of this paper describe the Hausdorff and upper box-counting dimensions in the case that the sequence $\bc$ doesn't switch too often between values. Namely, let $\kappa_n$ be the number of switches up to the time $n$: 
    \[
        \kappa_n:=\# \{ m<n \mid c_m \neq c_{m+1} \}.
    \]
We will say that the sequence $\bc$ \emph{switches rarely}, if $\kappa_n=o(n)$.

Also, given the sequence $\bc$, we consider the frequencies, with which indices $j$ occur in this sequence,
\[
\Freq_{n,j}:=\frac{\#\{m\le n\mid c_m=j\}}{n};
\]
and we will say that the sequence $\bc$ satisfies the \emph{frequencies condition}, if 
\begin{equation}\label{eq:freq}
    \forall j=1,\dots,k \quad \limsup_{n\to\infty} \, \Freq_{n,j}=1,        
\end{equation}
where $k$ is the size of the set of systems~$\mF$.

\begin{example}
    The sequences~$\bc$ that satisfy both rarely switching and frequencies assumptions exist. For instance, the sequence 
    \[
            \bc=
            \underbrace{00\dots 0}_{l_1\text{ times}}\,
            \underbrace{11\dots 1}_{l_2\text{ times}}\,
            \underbrace{00\dots 0}_{l_3\text{ times}}\,
            \underbrace{11\dots 1}_{l_4\text{ times}}\dots
    \]
    satisfies both assumptions once $\frac{l_{j+1}}{l_j}\to\infty$ as $j\to\infty$.
\end{example}

\begin{theorem}\label{t:main}
    Let $\mF=\{F_1,\dots,F_k\}$ be a system of $C^{1+\ha}$ expanding cookie-cutter maps, where $\ha>0$. Assume that the sequence $\bc=(c_n)$ 
    is rarely switching, $\kappa_n=o(n)$, and
    satisfies the frequencies condition~\eqref{eq:freq}.
    Then the Hausdorff and upper box-counting dimensions of the associated non-stationary cookie-cutter-like set $\Cantor_{\bc}=\Cantor(\mF;\bc)$ are given by
    \begin{equation}\label{eq:dimH-min}
        \dimH \Cantor_{\bc} = \min_{j=1,\dots,k} \dim \Cantor(F_j),  
    \end{equation}
    \begin{equation}\label{eq:dimB-max}
        \dimBu\Cantor_{\bc}= \max_{j=1,\dots,k} \dim \Cantor(F_j).
    \end{equation}    
\end{theorem}

\begin{remark}
    In the particular case, when all the maps forming each of the cookie-cutters~$F_j$ are \emph{affine}, the rarely switching condition~$\kappa_n=o(n)$ can be omitted. On the other hand, this assumption cannot be omitted in general, as we will see from Example~\ref{ex:HD-small} below. Also, note that the applicability of Theorem~\ref{t:main} can be sometimes extended by grouping letters into groups. For instance, the sequence 
    \[
      \bc = (00)^{l_1} (11)^{l_2} (01)^{l_3} \dots 
      (00)^{l_{3m+1}} (11)^{l_{3m+2}} (01)^{l_{3m+3}}, \dots
    \]
    (where the powers represent the number of repetitions and the products represent the concatenation of words),
    used in Example~\ref{ex:HD-small} below, does not satisfy the rarely switching assumption, but it does satisfy it once we group letters into pairs. Finally, as we will see at the end of Section~\ref{s:proofs} (see Proposition~\ref{p:frequencies} therein), the frequencies condition can be omitted, at the cost of conclusions~\eqref{eq:dimH-min} and~\eqref{eq:dimB-max} becoming more complicated, involving lower and upper limits respectively.
\end{remark}

\subsection{Implications for parameter-dependent systems}\label{s:parameter}

Consider now families of cookie-cutter maps/IFSs that depend on a parameter~$\prm$:
\[
\mF^{\prm}=\{F^{\prm}_1,\dots,F^{\prm}_k\}.
\]

From D.~Ruelle~\cite{R1982}, we know that the dimension of cookie-cutter sets of analytic cookie-cutter maps, depending analytically on a parameter, depend analytically on the same parameter. 
\begin{theorem}[D. Ruelle,~{\cite[Corollary~5]{R1982}}]\label{t: Ruelle}
    Let $J^\prm\subset [0,1]$ be the cookie-cutter set for a real analytic cookie-cutter map $F^\prm$, depending analytically on the parameter~$\prm$. Then, the Hausdorff dimension of $J^\prm=J(F^\prm)$ depends analytically on~$\prm$.
\end{theorem}

 However, even though such an analytic dependence takes place for individual (stationary) cookie-cutter sets, we see that it might not take place for the associated (non-stationary) cookie-cutter-like sets. Namely, applying Theorem~\ref{t:main} individually at each parameter~$\prm$, we get

\begin{corollary}\label{c:param}
    Assume that a family $\mF^{\prm}$ of parameter-dependent cookie-cutter sets satisfies the assumptions of Theorem~\ref{t:main} for each parameter~$\prm$, and that the sequence $\bc=(c_n)$ satisfies the rarely switching condition~$\kappa_n=o(n)$, as well as 
    the frequencies condition~\eqref{eq:freq}.
    Then for the cookie-cutter-like sets $\Cantor^{\prm}_{\bc}:=\Cantor(\mF^{\prm};\bc)$, one has
    \begin{equation}
        \dimH \Cantor^{\prm}_{\bc}= \min_{j=1,\dots,k}\dim \Cantor(F^{\prm}_j),
    \end{equation}
    \begin{equation}
        \dimBu \Cantor^{\prm}_{\bc} = \max_{j=1,\dots,k}\dim \Cantor(F^{\prm}_j).
    \end{equation}
\end{corollary}

Now, the minimum and maximum of two analytic functions lose differentiability at any point where their graphs have a transverse intersection (see Fig.~\ref{fig:dimensions}). We thus get an example of non-differentiable dependence of the Hausdorff and upper-box counting dimensions on the parameter:

\begin{theorem}\label{t:not-diff}
    Let the sequence~$\bc$ be as in Corollary~\ref{c:param}. 
    Assume that $F_1^{\prm}$, $F_2^{\prm}$ are two IFSs depending on a parameter~$\prm$, are $C^2$-smooth in both the space and parameter variables, and such that the graphs of the functions $\dim \Cantor(F^{\prm}_j), \, j=1,2$, intersect transversely at some parameter~$\prm_0$. 

    Then the dependence of Hausdorff and upper-box counting dimensions $\dimH$, $\overline{\dim}_B$ of the cookie-cutter-like set $\Cantor^{\prm}_{\bc}$ is not differentiable at $\prm_0$.
\end{theorem}

\begin{figure}[!h!]
    \centering
    \includegraphics[width=0.45\linewidth]{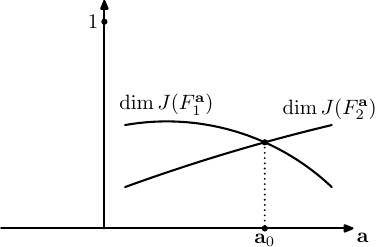} 
    \\
    \includegraphics[width=0.45\linewidth]{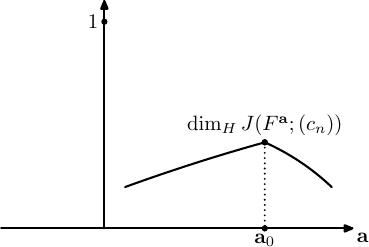} 
    \qquad 
    \includegraphics[width=0.45\linewidth]{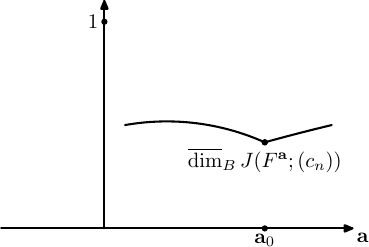} 
    \caption{Intersecting  graphs of the dimensions $\dimH(\Cantor(F_j^{\prm}))$ and the corresponding non-differentiable graphs of $\dimH$ and $\overline{\dim}_B$. }
    \label{fig:dimensions}
\end{figure}

\begin{remark}
    Slightly altering the above construction might lead to more non-differentiability points being created. Namely, if there are more than two systems, $F_1^{\prm},\dots,F_k^{\prm}$, then the graph of the minimum of their Hausdorff dimension is likely to have more non-differentiability points. Secondly, even if there are only two initial systems, $F_1^{\prm}$ and $F_2^{\prm}$, one can consider (in the same way as in Example~\ref{ex:F0-F1} below) their compositions of any fixed length, that adds new graphs to the list. For instance, using a sequence $\bc$ that satisfies rare switching and frequencies conditions after grouping letters into \emph{pairs} leads to the minimum of three graphs, those corresponding to the dimensions of $\Cantor(F^{\prm}_0)$, $ \Cantor(F^{\prm}_1)$ and $\Cantor(F^{\prm}_0 F^{\prm}_1).$
\end{remark}

\begin{question}
    How irregular can the graph of $\dim_H \Cantor^{\prm}_{\bc}$ be? Can it be non-differentiable on the right (or left) at some point? Can it have more than a countable number of non-differentiability points?
\end{question}

\section{Discussion: dimensions of the spectra of Sturmian Hamiltonians}\label{s:discussion}

Our motivation to study such (parameter-dependent) systems comes from the study of the spectra of Sturmian Hamiltonian operators. That is, the class of discrete Schr\"odinger operators, acting on $\ell_2(\Z)$ via 
\[
(H_{\lambda, \fa, x}u)_n:= u_{n+1}+u_{n-1}+ V_{\lambda, \fa, x}(n) u_{n}
\]
with a quasi-periodic potential 
\[
V_{\lambda, \fa, x}(n):=\lambda\cdot \chi_{[1-\fa, 1)}\left(x+n\fa \ (\text{mod} \ 1)\right),
\]
where the \textit{frequency} $\fa\in(0,1)$ is irrational, the \textit{coupling constant} $\prm:=\lambda$ is positive, and the \textit{phase}~$x$ is an element of $S^1$ (actually, the spectral properties of the operator are the same for Lebesgue-a.e. phase~$x$; see, e.g.~\cite[Chapter~4, Theorem~4.9]{DF2022}).

The properties of the spectrum $\Sigma_{\lambda,\fa}:=\sigma(H_{\lambda,\fa,x})$ of the Sturmian Hamiltonian have been extensively studied by many authors \cite{BBL2024, bist, C1986, Ca2009, CQ2025, DG2009, DG2011, DG2015, DGY2014, LW2004, LPW2007, LQW2014, M2014, L2025, Ra1997}; see also~\cite{DF2024} for a survey. One of the standard approaches here is to study the action of the corresponding \emph{renormalization maps}, acting on the corresponding \emph{Markov surface}~$S_{\lambda}$.

Namely, to every energy $E$ one puts in correspondence a point $p(E)$ of the Markov surface~$S_{\lambda}$, and a particular family $T_a:S_\lambda\to S_\lambda$, $a\in\mathbb N$, of renormalization diffeomorphisms. Now, given the decomposition of the frequency~$\fa$ in the continued fraction,
\[
\fa=[0;a_1,a_2,a_3,\dots],
\]
one considers the corresponding (non-stationary) sequence of renormalizations
\[
\Phi_n:=T_{a_n}\circ \dots \circ T_{a_1} : S_\lambda\to S_\lambda.
\]
The main connection is that an energy $E$ belongs to the spectrum $\Sigma_{\lambda,\fa}$ if and only if the orbit $(\Phi_n(p(E)))$ of the corresponding point $p(E)$ stays bounded~\cite{bist, D2000} (see also~\cite[Chapter~10, Theorem~10.5.2]{DF2024}).

For every $\fa$, one can consider the dimension $\dimH \Sigma_{\fa,\lambda}$ of the spectrum 
as a function of the coupling constant~$\lambda$. If $\fa$ is a quadratic irrationality (and hence its continued fraction is eventually periodic), the above bounded orbit criterion is reduced to the iterations of a single map. Such a renormalization map is known to be hyperbolic \cite{C1986, Ca2009, DG2009}. The spectrum is thus given by the intersection of the stable lamination of the maximal hyperbolic set $\Lambda_{\lambda}$ of this map with the transverse curve of initial conditions~$\{p(E)\}$. 

Meanwhile, the following result was established by M.~Pollicott~\cite{P2015}:

\begin{theorem}[M.~Pollicott,~{\cite{P2015}}]\label{t:Pollicott} Let M be a surface and $\Lambda\subset M$ a hyperbolic basic set for an analytic diffeomorphisms $f:M\rightarrow M$.
Suppose that $f_\prm:M\rightarrow M$, $f_0=f$, $\prm\in(-\epsilon, \epsilon)$ is a one-parameter family of analytic diffeomorphisms, depending analytically on $\prm$, with corresponding hyperbolic basic sets $\Lambda_\prm$, $\Lambda_0=\Lambda$, guaranteed by structural stability. Then, the stable dimension of $\Lambda_\prm$ is an analytic function of the parameter~$\prm$.
\end{theorem}

Hence, the dimension of the spectrum $\Sigma_{\fa,\lambda}$, being the stable dimension of $\Lambda_{\lambda}$, depends analytically on the coupling constant $\lambda$ as a parameter. 

Also, it was shown in~\cite[Theorem~1.2]{DG2015} that for any $\lambda\ge 24$ there exist numbers $D_H(\lambda)$, $\overline{D}_B(\lambda)$ such that 
\[
\dimH \Sigma_{\fa,\lambda} = D_H(\lambda), \quad 
\dimBu \Sigma_{\fa,\lambda} = \overline{D}_B(\lambda)
\]
for Lebesgue-a.e. frequency~$\fa$. In \cite{CQ2025}, it was shown (see Theorem~1.1 therein) that $D_H(\lambda)=\overline{D}_B(\lambda)$, and it was conjectured (see Sec.~1.4 therein) that these almost-sure dimensions are analytic functions of~$\lambda$.

For a generic $\fa$, all the maps $T_a$ are believed to satisfy a common stable cone condition~\cite{GJK}, which should lead to the existence and reasonable regularity of the non-stationary stable foliation. This is currently being done for irrationals of bounded type~\cite{JL} in the same spirit as it has been done for Anosov maps of the two-torus satisfying a common cone condition~\cite{Luna1}.

This setting is thus parallel to cookie-cutter-like sets, generated by non-stationary compositions of maps, depending on a parameter.

Now, consider graphs of the dimensions of spectra $\dim_H \Sigma_{\fa,\lambda}$ for all possible quadratic irrationalities~$\fa$. As there is countably many of such irrationalities (even with a uniform bound on the elements of the continued fraction), it is highly plausible that the graphs will not stay disjoint from each other: 
\begin{conjecture}\label{conj:quadratic}
    There exist quadratic irrationalities $\fa_1, \fa_2$ such that the graphs of the corresponding dimension functions  
    $\lambda \mapsto \dim_H \Sigma_{\fa,\lambda}$
    (transversely) intersect at some point~$\lambda_0$.    
\end{conjecture}
 
This, together with the (conjectural) generalization of  Theorem~\ref{t:not-diff} to this jointly-hyperbolic setting, would lead to the following example, that would be in a drastic contrast with the analytic dependence known for quadratic irrationalities $\alpha$ and conjectured for Lebesgue-a.e. $\alpha$:
\begin{conjecture}
    There exists a frequency $\fa$ of bounded type, such that the corresponding dimension function $\lambda \mapsto \dim_H \Sigma_{\fa,\lambda}$
    is not differentiable at at least one point~$\lambda_0$.
\end{conjecture}

So far, it is only known that the Hausdorff and upper-boxing counting dimensions of $\Sigma_{\fa,\lambda}$ are Lipschitz in $\lambda$ over bounded intervals of $[24, \infty)$; see~\cite{LQW2014}. 

\section{Moran formula and thermodynamic formalism for cookie-cutters}\label{s:Moran}

For convenience, from now on we will assume that the initial interval~$I$ is $[0,1]$; this will allow us to avoid~$|I|$ in the denominators.

\subsection{Cookie-cutter sets}

For a single cookie-cutter map $F$, for which the corresponding maps $f_i$ are affine,
\[
 f_j(x) = a_j x + b_j, \quad j=1,\dots,q,
\]
the dimension of the corresponding cookie-cutter set $\Cantor(F)$ is given by the so-called~\emph{Moran formula} (see~\cite[Theorem~II]{Mo1946}): one has $\dimH \Cantor(F) = s$, where $s$ is the solution to the equation
\begin{equation}\label{eq:Moran}
    \sum_{j=1}^q |a_j|^s =1.
\end{equation}
This formula can be seen as an equality between the $s$-dimensional volume of the set $\Cantor(F)$ and the sum of volumes of its parts~$f_j(\Cantor(F))$.

When the maps $f_i$ are no longer affine, but only contracting $C^{1+\ha}$-diffeomorphisms, the dimension of the set $\Cantor(F)$ is given by the \emph{Bowen formula}. Namely, for every $n$ and every word 
\[
w=w_1\dots w_n\in \Omega_n:= \{1,\dots,q\}^n,
\]
one considers the interval 
\[
I_w:=f_{w_1}\circ \dots \circ f_{w_n}(I).
\]
The \emph{partition function} of $n$ iterates is then defined as 
\[
Z_{n,F}(s) := \sum_{w\in\Omega_n} |I_w|^s,
\]
and the associated $n$-iterates \emph{pressure function} as 
\[
P_{n,F}(s):=\frac{1}{n} \log Z_n(s).
\]
This sequence of functions converges to a continuous decreasing \emph{stationary pressure function} 
\[
    P_F(s):=\lim_{n\to\infty} P_{n,F}(s).
\]
The \emph{Bowen formula} then states (see~\cite[Lemma~10]{B1980}) that the dimension $s$ (both Hausdorff and box-counting) of the cookie-cutter set~$\Cantor(F)$ is given by the solution of the equation $P_F(s)=0$.

Note that for the case of a cookie-cutter map $F$ with affine maps $f_i$, one has $Z_{n,F}=Z_{1,F}^n$, and thus the pressure function $P_F(s)$ is simply the logarithm of the left hand side of Moran's formula~\eqref{eq:Moran}.

\subsection{Thermodynamic formalism for cookie-cutter-like sets}

Assume that we are given a finite set $\mF=
\{F_1,\dots,F_k\}$ of $C^{1+\ha}$-cookie-cutter maps and a sequence $\bc=(c_n)_{n\in\N}$, taking values in $\{1,\dots,k\}$. To these data corresponds a cookie-cutter-like set $\Cantor(\mF;\bc)$, defined by~\eqref{eq:non-st-Cantor}; its Hausdorff and box-counting dimensions are then determined (see \cite{MRW2001,MU2003,RU2016}) via a generalization of the Bowen formula. 

Namely, let $(f_{j,1},\dots, f_{j,q_j})$ be the IFS corresponding to the cookie-cutter map $F_j$. For any sequence $c_1,\dots, c_n\in\{1,\dots,k\}$ define
\[
\Omega_{c_1\dots c_n} := \left\{w=w_1\dots w_n \mid \forall j=1,\dots,n \quad w_j \in \{1,\dots,q_{c_j}\} \right\}
\]
and let 
\[
I_{w_1\dots w_n}^{c_1\dots c_n}:= f_{c_1,w_1}\circ \dots \circ f_{c_n,w_n}(I), \quad w_1\dots w_n\in \Omega_{c_1\dots c_n},
\]
be the intervals that are connected components of the associated set $K_{c_1\dots c_n}$ (see~\eqref{eq:K-c-n}).

One then defines~\cite{MRW2001,MU2003,RU2016} non-stationary versions of the partition and pressure functions by
\begin{equation}\label{eq:def-PF}
Z_{c_1\dots c_n}(s):=\sum_{w_1\dots w_n\in \Omega_{c_1\dots c_n}} |I_{w_1\dots w_n}^{c_1\dots c_n}|^s,    
\end{equation}
\[
P_{c_1\dots c_n}(s):=\frac{1}{n} \log Z_{c_1\dots c_n}(s).
\]
Note that (in the same way as for stationary cookie-cutters) as there is a finite number of contracting diffeomorphisms $f_{i,j}$, there exist constants $1<\lambda<L$, such that for all $(i,j)$, one has
\[
\forall x\in I \quad L^{-1}<Df_{i,j}(x)<\lambda^{-1}.
\]
This implies the uniform exponential bounds for the intervals 
\[
L^{-n} < |I_{w_1\dots w_n}^{c_1\dots c_n}| < \lambda^{-n};
\]
hence, each summand in the definition~\eqref{eq:def-PF} of the partition function satisfies the estimates
\begin{equation}    
\forall r>0, \,\forall s, \quad |I_{w_1\dots w_n}^{c_1\dots c_n}|^{s} L^{-nr} < |I_{w_1\dots w_n}^{c_1\dots c_n}|^{s+r} < |I_{w_1\dots w_n}^{c_1\dots c_n}|^{s} \lambda^{-nr};
\end{equation}
adding them up and taking the logarithm, one gets the uniform decrease and continuity estimates
\begin{equation}\label{eq:slope}
P(s)-(\log L)\cdot r <P(s+r)< P(s)-(\log \lambda)\cdot r.
\end{equation}

In particular, there exists a unique zero $s_{c_1\dots c_n}$ of the function $P_{c_1\dots c_n}(s)$, and it is easy to see that it belongs to $[0,1]$. Now, let
\begin{equation}
    s_*:=\liminf_{n\to\infty} s_{c_1\dots c_n} \quad \text{and} \quad s^*:=\limsup_{n\to\infty} s_{c_1\dots c_n}.
\end{equation}
 Then, the dimensions of the cookie-cutter-like set $\Cantor_{\bc}=\Cantor(\mF;(c_n))$ are given by the following theorem.
\begin{theorem}[{\cite[Theorem~1.3]{MRW2001}}]\label{t:dimensions-roots}
    In the notations above,
    \begin{equation}\label{eq:dim-formulae}
        \dimH \Cantor_{\bc}=s_* \quad \text{and} \quad 
        \dimBu \Cantor_{\bc}=s^*.
    \end{equation}
\end{theorem}

\subsection{Distortion control}

We conclude this section by recalling some standard general statements of the distortion control. Namely, given a $C^1$-diffeomorphism $f:I\to I'$, one defines its \emph{distortion} on $I$ as
\[
\Dist(f;I):=\osc_{x\in I} \, \log |Df(x)|.
\]
The distortion is composition-subadditive:
\[
\Dist(f_1\circ f_2;I) \le \Dist(f_2;I) + \Dist(f_1;f_2(I)).
\]
At the same time, for a diffeomorphism $f$ of class $C^{1+\ha}$, its distortion on a subinterval $I'\subset I$ can be estimated as
\[
\Dist(f;I') \le C_f |I'|^{\ha},
\]
for some constant~$C_f$. A standard corollary to these two facts is that the distortion of compositions of $C^{1+\ha}$-contracting diffeomorphisms stays uniformly bounded. Namely, consider the compositions that occur in the non-stationary IFSs describing the cookie-cutter-like sets, given by 
\[
f_{w_1\dots w_{n}}^{c_1\dots c_{n}}:=f_{c_1,w_1}\circ \dots \circ f_{c_n,w_n}.
\]
We then have the following estimate.
\begin{lemma}\label{l: bounded distortion}
    There exists a constant $C_{\Dist}$, such that for any $c_1,\dots,c_n$ and any word $w\in \Omega_{c_1\dots c_n}$, one has 
    \begin{equation}\label{eq:dist-bound}
        \Dist(f_{w_1\dots w_{n}}^{c_1\dots c_{n}}; I)   
        \le C_{\Dist}.        
    \end{equation}
\end{lemma}
\begin{proof}
    Indeed, denoting $C_{\mF}:=\max_{i,j} C_{f_{i,j}}$, we see that due to the subadditivity of the distortion, the left hand side of~\eqref{eq:dist-bound} doesn't exceed the sum of the distortions of the individual maps on the corresponding intervals. This sum can be estimated as
    \[
        \sum_{m=0}^{n-1} \max_{i,j} (C_{f_{j,i}}) \cdot |I_{w_1\dots w_m}|^{\ha} \le C_{\mF} \sum_{m=0}^{\infty} \lambda^{-\ha m}= \frac{C_{\mF}}{1-\lambda^{-\ha}}=:C_{\Dist},
    \]
    where we have used $|I_{w_1\dots w_m}|\le \lambda^{-m}$.
\end{proof}

\section{Proofs}\label{s:proofs}
\subsection{Controlling additivity: main proposition}

The following proposition states that in ``rare switching'' mode, the non-stationary pressure function can be approximated by a linear combination of the stationary ones:
\begin{proposition}\label{p:combine}
   If the sequence $(c_n)$ satisfies the rare switching condition $\kappa_n=o(n)$, then the difference
    \begin{equation}
        P_{c_1\dots c_n}(s)- \sum_{j=1}^k \Freq_{n,j} \cdot P_{F_j}(s),
    \end{equation}
    where $P_{F_j}(s)$ are the corresponding stationary pressure functions, converges to zero as $n\to\infty$, and for any compact interval $[a,b]\subset \R$ this convergence is uniform in $s\in [a,b]$.
\end{proposition}

This proposition follows from the quasi-additivity of the log-partition functions: 

\begin{lemma}\label{l: almost sub-additive}
For any $s$ and any $c_{1},\dots, c_{n+m}$, one has
    \begin{equation}
        \left| \log Z_{c_1\dots c_{n+m}}(s)-\left(\log Z_{c_1\dots c_n}(s)+\log Z_{c_{n+1}\dots c_{n+m}}(s)\right)\right|\le C_{\Dist} \cdot |s|,
    \end{equation}
    where $C_{\Dist}$ is the constant from Lemma~\ref{l: bounded distortion}. 
\end{lemma}
\begin{proof}
    We have 
    \begin{equation}\label{eq:lenghts}
        I_{w_1\dots w_{n+m}}^{c_1\dots c_{n+m}}=f_{w_1\dots w_n}^{c_1\dots c_{n}} (I_{w_{n+1}\dots w_{n+m}}^{c_{n+1}\dots c_{n+m}}), \quad
        I_{w_1\dots w_{n}}^{c_1\dots c_{n}} = f_{w_1\dots w_n}^{c_1\dots c_{n}} (I).        
    \end{equation}
    Due to the mean value theorem, for a diffeomorphism $f:I'\to I''$ one has $|I''|=|Df(x)| \cdot |I'|$ for some $x\in I'$. Applying it to both images in~\eqref{eq:lenghts}, we get 
    \[
        |I_{w_1\dots w_{n+m}}^{c_1\dots c_{n+m}}| = |Df_{w_1\dots w_n}^{c_1\dots c_{n}} (x)| \cdot |I_{w_{n+1}\dots w_{n+m}}^{c_{n+1}\dots c_{n+m}}|,
    \]
    \[
        |I_{w_1\dots w_{n}}^{c_1\dots c_{n}}| = |Df_{w_1\dots w_n}^{c_1\dots c_{n}} (y)| \cdot |I|
    \]
    for some $x,y\in I$. Now, the quotient of derivatives of $f_{w_1\dots w_n}^{c_1\dots c_{n}}$ at any two points $x,y\in I$ doesn't exceed $e^{C_{\Dist}}$ due to Lemma~\ref{l: bounded distortion}, and as $|I|=1$, we have 
    \[
        e^{-C_{\Dist}} < \frac{|I_{w_1\dots w_{n+m}}^{c_1\dots c_{n+m}}|}{|I_{w_1\dots w_{n}}^{c_1\dots c_{n}}|\cdot |I_{w_{n+1}\dots w_{n+m}}^{c_{n+1}\dots c_{n+m}}|} < e^{C_{\Dist}}.
    \]
    Taking $s$-th power, multiplying by the denominator and summing over all the words $w$, we get 
    \begin{multline*}
        e^{-sC_{\Dist}} Z_{c_1\dots c_n}(s) \cdot Z_{c_{n+1}\dots c_{n+m}}(s) \le 
        Z_{c_1\dots c_{n+m}}(s) \\ \le 
        e^{sC_{\Dist}} Z_{c_1\dots c_n}(s) \cdot Z_{c_{n+1}\dots c_{n+m}}(s)
    \end{multline*}
    for $s\ge 0$, and the reversed inequalities for $s<0$. An application of the logarithm completes the proof.
\end{proof}

\begin{proof}[Proof of Proposition~\ref{p:combine}]
    Given any sequence $c_1,\dots,c_n$, cut it into consecutive groups of identical letters: let 
    \[
        c_1\dots c_n = (j_1)^{l_1} \dots (j_{\kappa_n+1})^{l_{\kappa_n+1}},
    \]
    where $j_1,\dots,j_{\kappa_n+1}\in \{1,\dots,k\}$ are symbols, and powers signify number of repetitions. Applying consecutively Lemma~\ref{l: almost sub-additive}, we get
    \begin{equation}\label{eq:c-cut}
        \left| \log Z_{c_1\dots c_n}(s) - \sum_i \log Z_{(j_i)^{l_i}}(s) \right| \le \kappa_n \cdot C_{\Dist}|s|.        
    \end{equation}
    Now, let us group the identical letters together: let 
    \[
        n_j:=\#\{m\le n \mid c_m=j\} = \sum_{i: \, j_i=j} l_i,
    \]
    then 
    \begin{equation}\label{eq:c-glue}
        \left| \sum_i \log Z_{(j_i)^{l_i}}(s) - \sum_{j=1}^k \log Z_{(j)^{n_j}}(s) \right| \le \kappa_n \cdot C_{\Dist}|s|.        
    \end{equation}
    Finally, note that for every $j$, $l$, $m$, one has 
    \[
        \left| m\log Z_{(j)^l} (s) - \log Z_{(j)^{ml}} (s) \right| \le m\cdot C_{\Dist}|s|;
    \]
    dividing by $m$ and passing to the limit as $m\to\infty$ implies that 
    \[
        \left| \log Z_{(j)^l} (s) - l \cdot P_{F_j}(s) \right| \le C_{\Dist}|s|.
    \]
    Joining this with~\eqref{eq:c-cut} and~\eqref{eq:c-glue}, we see that 
    \[
        \left| \log Z_{c_1\dots c_n}(s) - \sum_{j=1}^k n_j P_{F_j}(s) \right| \le (2\kappa_n+k) \cdot C_{\Dist}|s|.        
    \]
    Dividing by $n$ and using the rarely switching assumption $\kappa_n=o(n)$ completes the proof.
\end{proof}

\subsection{Estimating the dimension}

\begin{proof}[Proof of Theorem~\ref{t:main}]
We are going to use the result from~\cite{MRW2001}, Theorem~\ref{t:dimensions-roots} cited earlier. Namely, the Hausdorff and upper box-counting dimensions are, respectively, the lower and upper limits of the zeros~$s_{c_1\dots c_n}$ of the pressure functions $P_{c_1\dots c_n}(s)$ (see~\eqref{eq:dim-formulae}).

Now, note that due to the uniform lower bound~\eqref{eq:slope} for the slope of $P_{c_1\dots c_n}(s)$, for every $s', s''$, we have 
\[
|P_{c_1\dots c_n}(s') - P_{c_1\dots c_n}(s'')| \ge \log \lambda \cdot |s' - s''|
\]
and hence
\[
|s_{c_1\dots c_n} - s'| \le \frac{1}{\log \lambda} |P_{c_1\dots c_n}(s')|.
\]

Now, for every $n$, let $s'_{c_1\dots c_n}$ be the root of the linear combination of pressure functions 
\[
    \sum_{j=1}^k  \Freq_{n,j} \cdot P_{F_j}(s).
\]
Then Proposition~\ref{p:combine} implies that the value $P_{c_1\dots c_n}(s'_{c_1\dots c_n})$ tends to zero, and hence the same applies to the difference
$|s_{c_1\dots c_n}-s'_{c_1\dots c_n}|$.
Applying Theorem~\ref{t:dimensions-roots} immediately leads to
\begin{equation}\label{eq:dim-s-prim}
    \dimH \Cantor_{\bc} = \liminf_{n\to\infty} s'_{c_1\dots c_n}, \quad 
    \dimBu \Cantor_{\bc} = \limsup_{n\to\infty} s'_{c_1\dots c_n}.
\end{equation}
Now, on one hand, due to the monotonicity of each of the pressure functions $P_{F_j}(s)$, for every $n$, the root $s'_{c_1\dots c_n}$ is comprised between the minimum and the maximum of the roots of these functions, that are exactly the minimum and the maximum of the dimensions of the sets~$\Cantor(F_j)$. Hence, 
\begin{equation}\label{eq:inequalities}    
\liminf_{n\to\infty} s'_{c_1\dots c_n} \ge \min_j \dim \Cantor(F_j), \quad \limsup_{n\to\infty} s'_{c_1\dots c_n} \le \max_j \dim \Cantor(F_j).
\end{equation}
On the other hand, if for some $n,j_0,\eps$ we have $\Freq_{n,j_0}\ge 1-\eps$, then for $s'=\dim \Cantor(F_{j_0})$ one has
\[
    \left|\sum_{j=1}^k  \Freq_{n,j} \cdot P_{F_j}(s)\right| \le \eps\cdot \max_{j} |P_{F_j}(s')|;
\]
taking $M:=\max_{s\in [0,1]} \max_j |P_{F_j}(s)|$, we thus have 
\[
|s_{c_1\dots c_n} - s'| \le \frac{M}{\log \lambda} \, \eps,
\]
and the frequencies condition then immediately implies that the inequalities in~\eqref{eq:inequalities} are actually equalities. Together with~\eqref{eq:dim-s-prim} this concludes the proof of Theorem~\ref{t:main}.
\end{proof}

We conclude this section by noticing that Theorem~\ref{t:main} can be stated in a way that does not require the frequency conditions. Namely, consider the simplex 
\[
\Delta:=\{(\alpha_1,\dots,\alpha_k)\mid \forall j \quad \alpha_j\ge 0, \quad \sum_j \alpha_j=1\}
\]
and a map $\Root:\Delta\to [0,1]$ that associates to $(\alpha_1,\dots,\alpha_k)$ the root of the corresponding linear combination of pressure functions:
\[
    \Root(\alpha_1,\dots,\alpha_k)=s: \quad \sum_j \alpha_j P_{F_j}(s)=0.
\]
This map is continuous due to the same slope arguments as before, and we have 
\[
    s'_{c_1\dots c_n} = \Root(\Freq_{n,1},\dots,\Freq_{n,k}).
\]
Hence, the following statement holds.
\begin{proposition}\label{p:frequencies}
    Assume that all the assumptions of Theorem~\ref{t:main} hold, except possibly for the frequency assumptions. Then the dimensions are given by the formulae
    \[
        \dimH \Cantor_{\bc} = \liminf_{n\to\infty} \, \Root(\Freq_{n,1},\dots,\Freq_{n,k}),
    \]
    \[
        \dimBu \Cantor_{\bc} = \limsup_{n\to\infty} \, \Root(\Freq_{n,1},\dots,\Freq_{n,k}).
    \]
\end{proposition}
In particular, for given cookie-cutter maps $F_1,\dots, F_k$, assuming that the sequence $(c_n)$ is rarely switching, the corresponding dimensions depend only on the set of accumulation points of the vector of frequencies in the simplex~$\Delta$.

\section{Lower Hausdorff dimension example}\label{s:lower}

This section is devoted to the construction of an example, showing that the assumption of rare switching in Theorem~\ref{t:main} cannot be omitted. We start with the following construction.

\begin{example}\label{ex:F0-F1}
    For every $\eps>0$, consider cookie-cutter maps $F_0,F_1$, where 
    \[
        F_0(x) = \begin{cases}
            f_0^{-1}(x), & x\in [0,\frac{1}{20}]\\
            f_1^{-1}(x), & x\in [\frac{19}{20},1],
        \end{cases}, \qquad F_1(x) = 1-F_0(1-x),
    \]
    the map $f_1^{-1}(x)=20x-19$ is the affine map from the interval $[\frac{19}{20},1]$ to $[0,1]$,
    and $g(x)=f_0^{-1}(x)$ is a M\"obius map, defined by 
    \[
        g(0)=0,\quad g'(0)=1+\eps, \quad g(\frac{1}{20})=1.
    \]
    Then, once $\eps>0$ is sufficiently small, we have 
    \[
        \dim \Cantor(F_0 F_1) < \frac{1}{2.01}
        < \dim \Cantor(F_0)=\dim \Cantor(F_1).
    \]
\end{example}

\begin{figure}
    \centering
    \includegraphics[width=0.35\linewidth]{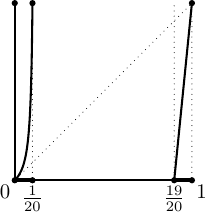} \qquad \qquad 
    \includegraphics[width=0.35\linewidth]{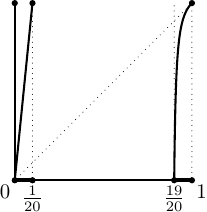} 
    \caption{Graphs of $F_0$,$F_1$}
    \label{fig:ex2}
\end{figure}

\begin{proof}
    As $\eps\to 0$, the lower limit of the dimension of the set $\Cantor(F_0)$ is at least~$1/2$. Indeed, for any fixed $M$ the same Cantor set is generated also by 
    \[
    f_1, f_0f_1, f_0^2 f_1,\dots, f_0^{M-1}f_1, f_0^{M}.
    \]
    However, for every $M$, for all $\eps$ sufficiently close to $0$, the derivatives of these maps satisfy a lower bound $Df>\frac{1}{500 M^2}$, thus implying a lower bound for the Hausdorff dimension:
    \[
        \dimH \Cantor(F_0) \ge \frac{\log M}{\log 500M^2} \to \frac{1}{2}.
    \]
    On the other hand, the IFS corresponding to~$F_0 F_1$ consists of four maps, and it is not difficult to check that each of these maps contracts at every point of $[0,\frac{1}{20}]\cup [\frac{19}{20},1]$ at least $20$ times, thus 
    \[
    \dim J(F_0 F_1)< \frac{\log 4}{\log 20} <\frac{1}{2.01}.
    \]
\end{proof}

\begin{example}\label{ex:HD-small}
    Take the sequence $\bc=(c_n)$ to be formed by consecutive groups of $00$'s, $11$'s and $01$'s, 
    \[
      \bc = (00)^{l_1} (11)^{l_2} (01)^{l_3} \dots 
      (00)^{l_{3m+1}} (11)^{l_{3m+2}} (01)^{l_{3m+3}} \dots
    \]
    (where the powers represent the number of repetitions and the products represent the concatenation of words).
    Assume that the lengths $l_j$ of these groups satisfy $\frac{l_{j+1}}{l_j} \to \infty$. Take any $\eps>0$ satisfying the conclusion of Example~\ref{ex:F0-F1}. Then the Hausdorff dimension of the corresponding cookie-cutter-like set satisfies
    \[
        \dim_H \Cantor_{\bc} = \dim_H \Cantor(F_0 F_1) < \dim_H \Cantor(F_0)=\dim_H \Cantor(F_1);
    \]
    in particular, the conclusion~\eqref{eq:dimH-min} of Theorem~\ref{t:main} is not satisfied for this set.
\end{example}

\begin{proof}
    It suffices to apply Theorem~\ref{t:main} for couples of consecutive letters: the system then satisfies both the rarely switching condition and the frequencies condition due to the assumption on the lengths~$l_j$. 
\end{proof}


\begin{thebibliography}{99}

\bibitem[BBL2024]{BBL2024} R. Band, S. Beckus, R. Loewy, The Dry Ten Martini Problem for Sturmian Hamiltonians, (2024). arXiv:2402.16703 

\bibitem[BIST1989]{bist} J. Bellissard,  B. Iochum, E. Scoppola,  D. Testard, Spectral properties of one-dimensional
equasicrystals, {\it Commun. Math. Phys.}, \textbf{125} (1989), pp. 527--543.

\bibitem[B1980]{B1980} R. Bowen, Hausdorff dimension of quasi-circles, \emph{Publ. Math. IHES}, \textbf{50} (1980), pp.~11--25

\bibitem[C1986]{C1986} M. Casdagli, Symbolic dynamics for the renormalization map
of a quasiperiodic Schr\"odinger equation, \emph{Comm. Math. Phys.}\ \textbf{107} (1986), pp.~295--318.


\bibitem[Ca2009]{Ca2009} S.~Cantat, Bers and H\'enon, Painlev\'e and Schr\"odinger, \textit{Duke Math.\ J.}\ \textbf{149} (2009), pp.~411--460.

\bibitem[MCM1983]{MCM1983}
H. McCluskey, A. Manning,
Hausdorff dimension for horseshoes,
\emph{Ergodic Theory Dynam. Systems} \textbf{3}:2 (1983), pp.~251--260.




\bibitem[CQ2025]{CQ2025} J.~Cao, Y.~Qu,
Almost sure dimensional properties for the spectrum and the density of states of Sturmian Hamiltonians,
\emph{Advances in Mathematics},
\textbf{478} (2025),
110387,
ISSN 0001-8708,
https://doi.org/10.1016/j.aim.2025.110387

\bibitem[D2000]{D2000} D. Damanik, Substitution Hamiltonians with bounded trace map orbits, {\it J. Math. Anal. Appl.}, \textbf{249}:2 (2000), pp. 393--411.

\bibitem[DG2009]{DG2009} Damanik D., Gorodetski A., Hyperbolicity of the Trace Map for the Weakly Coupled Fibonacci Hamiltonian,  {\it Nonlinearity,}  \textbf{22} (2009), pp. 123--143.

\bibitem[DF2022]{DF2022} D. Damanik and J. Fillman, {\it One-dimensional ergodic Schr\"odinger operators---I. General theory}, Graduate Studies in Mathematics, 221, Amer. Math. Soc., Providence, RI, 2022. 

\bibitem[DF2024]{DF2024} D. Damanik and J. Fillman, {\it One-dimensional ergodic Schr\"odinger operators---II. Specific classes}, Graduate Studies in Mathematics, 249, Amer. Math. Soc., Providence, RI, 2024.


\bibitem[DG2011]{DG2011} Damanik D., Gorodetski A., Spectral and Quantum Dynamical Properties of the Weakly Coupled Fibonacci Hamiltonian, {\it  Communications in Mathematical Physics}, \textbf{305} (2011), pp. 221--277.

\bibitem[DG2015]{DG2015} Damanik D., Gorodetski A., Almost Sure Frequency Independence of the Dimension of the Spectrum of Sturmian Hamiltonians, \textit{Communications in Mathematical Physics}, \textbf{337} (2015), pp. 1241--1253.

\bibitem[DGY2014]{DGY2014} Damanik D., Gorodetski A., Yessen W.N. The Fibonacci Hamiltonian. \emph{Inventiones mathematicae}, \textbf{206} (2014), pp.~629--692.

\bibitem[FRW1997]{FRW1997} D.-J. Feng, H. Rao and J. Wu, The net measure properties of symmetric Cantor sets and their applications, \emph{Progr. Natur. Sci. (English Ed.)} {\bf 7}:2 (1997), 172--178.


\bibitem[F2015]{2015} J.~M. Fraser, Remarks on the analyticity of subadditive pressure for products of triangular matrices, \emph{Monatsh. Math.} {\bf 177}:1 (2015),  53--65

\bibitem[GJK2025+]{GJK} Anton Gorodetski, Seung uk Jang, Victor Kleptsyn, Invariant Cone Field for Sturmian Trace Maps, in progress.




\bibitem[H1981]{Hutchinson} John E. Hutchinson,  Fractals and self similarity. \emph{Indiana Univ. Math. J.} \textbf{30}:5 (1981), pp.~713--747. 
https://doi.org/10.1512/iumj.1981.30.30055

\bibitem[JL2025+]{JL} Seung uk Jang, Alexandro Luna, work in progress




\bibitem[LW2004]{LW2004} Liu, QH., Wen, ZY. Hausdorff Dimension of Spectrum of One-Dimensional Schrödinger Operator with Sturmian Potentials. \textit{Potential Analysis}, \textbf{20} (2004), pp.~33--59. https://doi.org/10.1023/A:1025537823884

\bibitem[LPW2007]{LPW2007} QH. Liu, J. Peyrière, ZY. Wen. Dimension of the spectrum of one-dimensional discrete Schrödinger operators with Sturmian potentials. \emph{Comptes Rendus Mathématique}, \textbf{345}:12 (2007), pp. 667-672. https://doi.org/10.1016/j.crma.2007.10.048 

\bibitem[LQW2014]{LQW2014}
QH. Liu, YH. Qu, ZY. Wen,
The fractal dimensions of the spectrum of Sturm Hamiltonian,
\textit{Advances in Mathematics},
\textbf{257} (2014),
pp.~285--336,
https://doi.org/10.1016/j.aim.2014.02.019.

\bibitem[L2024]{Luna1} A. Luna, Regularity of Non-stationary Stable Foliations of Toral Anosov Maps, arXiv:2410.07406

\bibitem[L2025]{L2025} A. Luna. On the Spectrum of Sturmian Hamiltonians of Bounded Type in a Small Coupling Regime. \emph{Ann. Henri Poincaré} (2025). https://doi.org/10.1007/s00023-025-01581-z



\bibitem[MRW2001]{MRW2001} J. Ma, H. Rao. \& Z. Wen. Dimensions of cookie-cutter-like sets. \emph{Sci. China Ser. A-Math}, \textbf{44} (2001), pp.~1400--1412. https://doi.org/10.1007/BF02877068

\bibitem[Ma1990]{Ma1990} R. Ma\~{n}\'{e},
The Hausdorff dimension of horseshoes of diffeomorphisms of surfaces.
\emph{Bol. Soc. Brasil. Mat. (N.S.)} \textbf{20}:2 (1990), pp.~1–-24.

\bibitem[MU2003]{MU2003} D. Mauldin and M. Urba\'nski, \emph{Graph Directed Markov Systems: Geometry and Dynamics of Limit Sets}, Cambridge University Press (2003).
https://doi.org/10.1017/CBO9780511543050



\bibitem[MSU2013]{MSU2013} V. Mayer, B. Skorulski and M. Urba\'nski, Regularity and irregularity of fiber dimensions of non-autonomous dynamical systems, \emph{Ann. Acad. Sci. Fenn. Math.} {\bf 38}:2 (2013), pp.~489--514

\bibitem[MUZ2020]{MUZ2020} V. Mayer, M. Urba\'nski and A. Zdunik, Real analyticity for random dynamics of transcendental functions, \emph{Ergodic Theory Dynam. Systems} {\bf 40}:2 (2020), pp.~490--520



\bibitem[M2014]{M2014} M. Mei, Spectral properties of discrete Schr\"odinger operators with potentials generated by primitive invertible substitutions, \emph{J. Math. Phys}. \textbf{55}, 082701 (2014).

\bibitem[Mo1946]{Mo1946} P.A.P. Moran, Additive functions of intervals and Hausdorff measure. \emph{Mathematical Proceedings of the Cambridge Philosophical Society}, \textbf{42}:1(1946), pp.15--23. 
https://doi.org/10.1017/S0305004100022684

\bibitem[PT1993]{PT1993} J. Palis, F. Takens, Hyperbolicity and sensitive
chaotic dynamics at homoclinic bifurcations. \textit{Cambridge University
Press}, 1993.

\bibitem[P2015]{P2015} M. Pollicott, Analyticity of dimensions for hyperbolic surface diffeomorphisms, \emph{Proc. Amer. Math. Soc.} {\bf 143}:8, (2015), pp.~3465--3474.

\bibitem[RU2016]{RU2016} L. Rempe-Gillen and M. Urba\'nski, Non-autonomous conformal iterated function systems and Moran-set constructions, \emph{Trans. Amer. Math. Soc.} {\bf 368}:3 (2016), pp.~1979--2017.


\bibitem[Ra1997]{Ra1997}  L. Raymond, A constructive gap labelling for the discrete Schr\"odinger operator on a quasiperiodic chain, Preprint (1997).

\bibitem[R2020]{R2020} M.~K. Roychowdhury, Topological pressure and fractal dimensions of cookie-cutter-like sets, \emph{Southeast Asian Bull. Math.} {\bf 44}:5 (2020), 719--732

\bibitem[R1982]{R1982} D. Ruelle, Repellers for real analytic maps, \emph{Ergodic Theory Dynamical Systems} \textbf{2}:1 (1982), pp.~99–107.

\bibitem[R2008]{R2008} H.~H. Rugh, On the dimensions of conformal repellers. Randomness and parameter dependency, \emph{Ann. of Math. (2)}, {\bf 168}:3 (2008), pp.~695--748

\end{thebibliography}
\end{document}